\documentclass[11pt]{amsart}
\usepackage{epsfig,amssymb}

\theoremstyle{plain}
\newtheorem{theorem}{Theorem}[section]
\newtheorem{proposition}[theorem]{Proposition}

\newtheorem{lemma}[theorem]{Lemma}
\newtheorem{definition}[theorem]{Definition}

\title{\bf Parity condition for irreducibility\\ of Heegaard splittings}

\author{Jung Hoon Lee}

\address{\small School of Mathematics, KIAS\\
207-43, Cheongnyangni 2-dong, Dongdaemun-gu\\
Seoul, Korea\\} \email{jhlee@kias.re.kr}

\date{}

\begin{document}



\begin{abstract}
Casson and Gordon gave the rectangle condition for strong
irreducibility of Heegaard splittings \cite{CG}. We give a parity
condition for irreducibility of Heegaard splittings of irreducible
manifolds. As an application, we give examples of non-stabilized
Heegaard splittings by doing a single Dehn twist.
\end{abstract}

\maketitle

\section{Introduction}
 Given a non-minimal genus Heegaard splitting of a $3$-manifold, it is
 not an easy problem to show that it is irreducible or cannot be destabilized.
 In \cite{Kobayashi2}, Kobayashi showed that every genus $g\ge 3$
Heegaard splitting of $2$-bridge knot exterior is reducible. The
motivation of this paper was the question that whether there
exists an irreducible genus three Heegaard splitting of a tunnel
number one knot exterior which is not $2$-bridge.

 Casson and Gordon used the rectangle condition on Heegaard
 diagrams to show strong irreducibility of certain manifolds obtained
 by surgery in their unpublished paper \cite{CG}. See also (\cite{MS},
 Appendix).
 One can also refer to the paper by Kobayashi \cite{Kobayashi1}
 and Saito (\cite{Saito}, section $7$) for a good application of the
rectangle condition.

 In section $3$, we review that Casson-Gordon's rectangle
 condition implies strong irreducibility of Heegaard splittings.
 Also we consider weak version of rectangle condition for
 manifolds with non-empty boundary.

 In section $4$, we give a parity condition on Heegaard diagrams
 to guarantee that the given Heegaard splitting is non-stabilized.
 Hence, if the manifold under consideration is irreducible, the
 Heegaard splitting is irreducible.

 \begin{theorem}
 Suppose M is an irreducible $3$-manifold.  Let $H_1\cup_S H_2$
  be a Heegaard splitting of M.
  Let $\{D_1,D_2,\cdots,D_{3g-3}\}$ and $\{E_1,E_2,\cdots,E_{3g-3}\}$
  be collections of essential disks of  $H_1$ and $H_2$ respectively giving
 pants decomposition of $S$.

 If $|D_i\cap E_j|\equiv 0 \pmod{2}$ for all the pairs $(i,j)$,
 then $H_1\cup_S H_2$ is irreducible.
\end{theorem}

As an application, we give examples of non-stabilized Heegaard
splittings by doing a single Dehn twist, in section $5$.

\section{Pants decomposition and essential disk in a handlebody}

Let $H$ be a genus $g\ge 2$ handlebody. Suppose a collection of
$3g-3$ essential disks $\{D_1,D_2,\cdots,D_{3g-3}\}$ cuts $H$ into
a collection of $2g-2$ solid pants-shaped $3$-balls
$\{B_1,B_2,\cdots,B_{2g-2}\}$. Let $P_i$ be the pants $B_i\cap
\partial H$ ($i=1,2,\cdots,2g-2$). Then $P_1\cup P_2\cup\cdots
\cup P_{2g-2}$ is a {\it pants decomposition} of $\partial H$.

Let $D$ be an essential disk in $H$. Assume that $D$ intersects
$\underset{i=1}{\overset{3g-3}{\bigcup}} D_i$ minimally.

\begin{definition}
A wave $\alpha(D)$ for an essential disk $D$ and the collection of
essential disks $\{D_1,D_2,\cdots,D_{3g-3}\}$ of a handlebody $H$
is a subarc of $\partial D$ cut by the intersection $D\cap
(\underset{i=1}{\overset{3g-3}{\bigcup}} D_i)$ satisfying the
following conditions.

\begin{itemize}
\item There exist an outermost arc $\beta$ and outermost disk
$\Delta$ of $D$ with $\partial\Delta=\alpha(D)\cup\beta$.
 \item
$(\alpha(D),\partial\alpha(D))$ is not homotopic, in $\partial H$,
into $\partial D_i$ containing $\partial\alpha(D)$.
\end{itemize}
\end{definition}

\begin{lemma}
Suppose $D$ is not isotopic to any $D_i$ $(i=1,2,\cdots,3g-3)$.
Then $\partial D$ contains a wave.
\end{lemma}

\begin{proof}
Suppose $D\cap (\underset{i=1}{\overset{3g-3}{\bigcup}}
D_i)=\emptyset$. Then $\partial D$ lives in one of the pants $P_i$
for some $i$. Hence we can see that $D$ is isotopic to some $D_i$,
which contradicts the hypothesis of lemma. Therefore $D\cap
(\underset{i=1}{\overset{3g-3}{\bigcup}} D_i)\ne\emptyset$.

Consider the intersection $D\cap
(\underset{i=1}{\overset{3g-3}{\bigcup}} D_i)$.
 We can eliminate simple closed curves
of intersection in $D$ by standard innermost disk argument since
handlebodies are irreducible. So we may assume that the
intersection is a collection of arcs. Then there always exists an
outermost arc and outermost disk, hence a wave.
\end{proof}

The collection of arcs of intersection $D\cap
(\underset{i=1}{\overset{3g-3}{\bigcup}} D_i)$ divides $D$ into
subdisks. A subdisk would be a $2n$-gon such as bigon, $4$-gon,
and so on. Note that bigons are in one-to-one correspondence with
outermost disks, hence with waves. The following observation is
important for the parity condition that will be discussed in
section $4$.

\begin{lemma}
For all $i$ $(i=1,2,\cdots,3g-3)$, $|\partial D\cap \partial
D_i|\equiv 0 \pmod{2}$.
\end{lemma}

\begin{proof}
Since any arc of intersection of $D\cap D_i$ has two endpoints,
$|\partial D\cap \partial D_i|$ would be an even number.
\end{proof}

\section{Rectangle condition}
Let $S$ be a closed genus $g\ge 2$ surface and $P_1$ and $P_2$ be
pants in $S$ with $\partial P_i=a_i\cup b_i\cup c_i$ ($i=1,2$).
Assume that $\partial P_1$ and $\partial P_2$ intersect
transversely.

\begin{definition}
We say that $P_1$ and $P_2$ are tight if
\begin{enumerate}
\item There is no bigon $\Delta$ in $S$ with
$\partial\Delta=\alpha\cup\beta$, where $\alpha$ is a subarc of
$\partial P_1$ and $\beta$ is a subarc of $\partial P_2$
 \item For the pair $(a_1,b_1)$ and $(a_2,b_2)$, there is a
 rectangle $R$ embedded in $P_1$ and $P_2$ such that the interior
 of $R$ is disjoint from $\partial P_1\cup\partial P_2$ and the
 edges of $R$ are subarcs of $a_1,b_1,a_2,b_2$. This statement
 holds for the following all nine combinations of pairs.
 $$(a_1,b_1),(a_2,b_2)\quad (a_1,b_1),(b_2,c_2)\quad (a_1,b_1),(c_2,a_2)$$
 $$(b_1,c_1),(a_2,b_2)\quad (b_1,c_1),(b_2,c_2)\quad (b_1,c_1),(c_2,a_2)$$
 $$(c_1,a_1),(a_2,b_2)\quad (c_1,a_1),(b_2,c_2)\quad (c_1,a_1),(c_2,a_2)$$
\end{enumerate}
\end{definition}

Let $H_1\cup_S H_2$ be a genus $g\ge 2$ Heegaard splitting of a
$3$-manifold $M$. Let $\{D_1,D_2,\cdots,D_{3g-3}\}$ be a
collection of essential disks of $H_1$ giving a pants
decomposition $P_1\cup P_2\cup\cdots\cup P_{2g-2}$ of $S$ and
$\{E_1,E_2,\cdots,E_{3g-3}\}$ be a collection of essential disks
of $H_2$ giving a pants decomposition $Q_1\cup Q_2\cup\cdots\cup
Q_{2g-2}$ of $S$. Casson and Gordon introduced the rectangle
condition to show strong irreducibility of Heegaard splittings
\cite{CG}.  See also \cite{Kobayashi1}.

\begin{definition}
We say that $P_1\cup P_2\cup\cdots\cup P_{2g-2}$ and $Q_1\cup
Q_2\cup\cdots\cup Q_{2g-2}$ of $H_1\cup_S H_2$ satisfies the
rectangle condition if for each $i=1,2,\cdots,2g-2$ and
$j=1,2,\cdots,2g-2$, $P_i$ and $Q_j$ are tight.
\end{definition}

At a first glance, it looks not so obvious that rectangle
condition implies strong irreducibility. Here we give a proof.

\begin{proposition}
Suppose $P_1\cup P_2\cup\cdots\cup P_{2g-2}$ and $Q_1\cup
Q_2\cup\cdots\cup Q_{2g-2}$ of $H_1\cup_S H_2$ satisfies the
rectangle condition. Then it is strongly irreducible.
\end{proposition}

\begin{proof}
Suppose $H_1\cup _S H_2$ is not strongly irreducible. Then there
exist essential disks $D\subset H_1$ and $E\subset H_2$ with
$D\cap E=\emptyset$. Suppose there is a bigon $\Delta$ in $S$ with
$\partial\Delta=\alpha\cup\beta$, where $\alpha$ is a subarc of
$\partial D$ and $\beta$ is a subarc of $\partial P_i$ for some
$i$. If any subarc of $\partial E$ is in $\Delta$, we remove it by
isotopy into $S-\Delta$ before we remove the bigon $\Delta$ by
isotopy of neighborhood of $\alpha$ in $D$. So we can remove such
bigons maintaining the property that $D\cap E=\emptyset$. Also
note that the number of components of intersection $|\partial
E\cap (\underset{j=1}{\overset{3g-3}{\bigcup}} \partial E_j)|$
does not increase after the isotopy since there is no bigon
$\Delta'$ with $\partial\Delta'=\gamma\cup\delta$, where $\gamma$
is a subarc of $\partial P_i$ for some $i$ and $\delta$ is a
subarc of $\partial Q_j$ for some $j$ by the definition of
tightness of $P_i$ and $Q_j$. We can also remove a bigon made by a
subarc of $\partial E$ and a subarc of $\partial Q_j$ for some $j$
similarly. So we may assume that $D$ intersects
$(\underset{i=1}{\overset{3g-3}{\bigcup}} D_i)$ minimally and $E$
intersects $(\underset{j=1}{\overset{3g-3}{\bigcup}} E_j)$
minimally with $D\cap E=\emptyset$.

Suppose $D$ is isotopic to $D_i$ for some $i$. Let $\partial
Q_j=a_j\cup b_j\cup c_j$ ($j=1,2,\cdots,2g-2$). Then $\partial
D\cap Q_j$ contains all three types of essential arcs
$\alpha_{j,ab}, \alpha_{j,bc}, \alpha_{j,ca}$ by the rectangle
condition, where $\alpha_{j,ab}$ is an arc in $Q_j$ connecting
$a_j$ and $b_j$, $\alpha_{j,bc}$ is an arc connecting $b_j$ and
$c_j$ and $\alpha_{j,ca}$ is an arc connecting $c_j$ and $a_j$.
Then $E$ is not isotopic to any $E_j$ since $D\cap E=\emptyset$.
Then $\partial E$ contains a wave by Lemma 2.2 and this is a
contradiction since a wave intersects at least one of
$\alpha_{j,ab}, \alpha_{j,bc}, \alpha_{j,ca}$ for some $j$ and
$D\cap E=\emptyset$.

If $D$ is not isotopic to any $D_i$, $\partial D$ contains a wave
by Lemma 2.2. Then also in this case, $\partial D\cap Q_j$
contains all three types of essential arcs $\alpha_{j,ab},
\alpha_{j,bc}, \alpha_{j,ca}$ of $Q_j$ by the rectangle condition.
This gives a contradiction by the same argument as in the above.
\end{proof}

\begin{figure}[h]
   \centerline{\includegraphics[width=9.5cm]{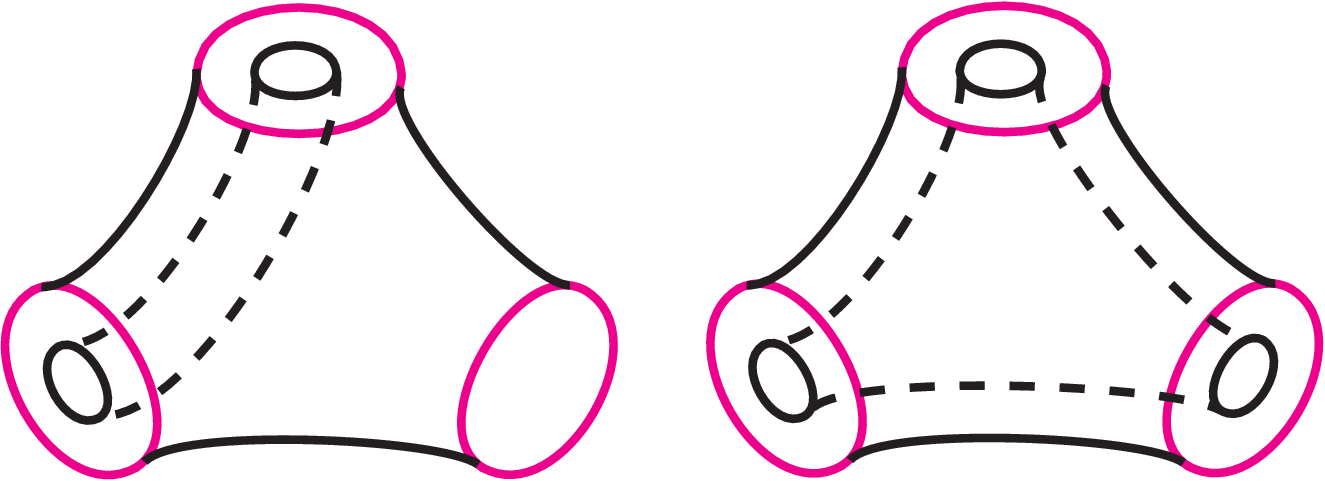}}
    \caption{Pants decomposition for compressionbody}
\end{figure}

 Now we consider manifolds with non-empty boundary. For a given
 compressionbody, there can be many ways of pants decomposition.
 But we can give a specific pants decomposition as in the Figure
 1. where if a piece has inner hole, it is one of the two types in
 the Figure 1.

\begin{figure}[h]
  \centerline{\includegraphics[width=10cm]{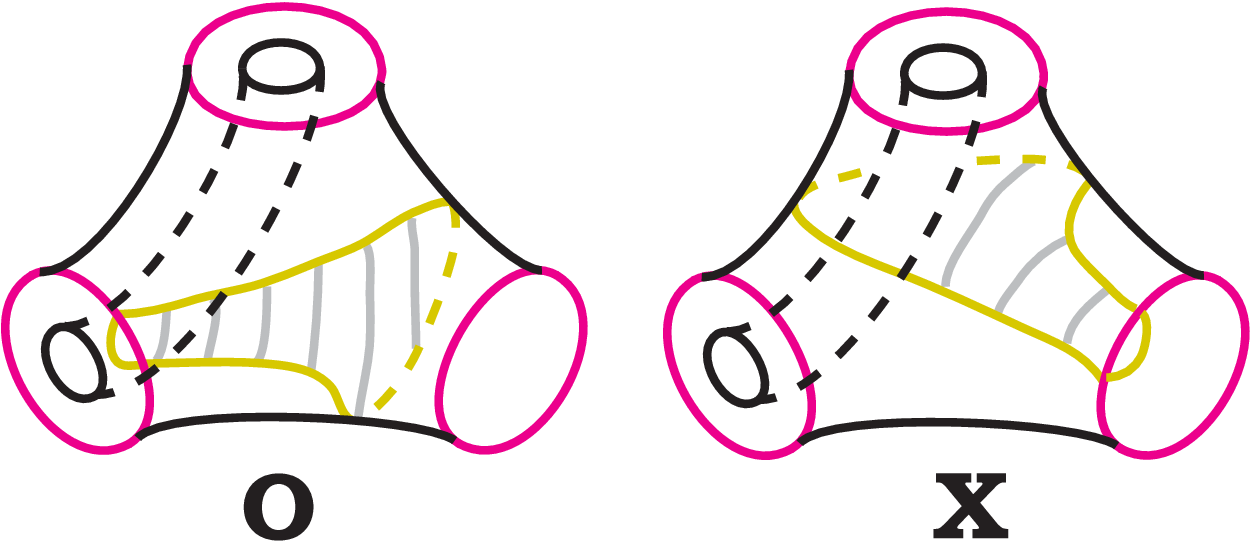}}
    \caption{Wave is more restrictive in compressionbody.}
\end{figure}

For an essential disk in a compressionbody, the existence of wave
is more restrictive because of the ``minus" boundary of
compressionbody (Figure 2.). So we can say the following
proposition.

\begin{proposition}
For a Heegaard splitting $H_1\cup_S H_2$ and pants decomposition
$P_1\cup P_2\cup\cdots\cup P_{2g-2}$ and $Q_1\cup
Q_2\cup\cdots\cup Q_{2g-2}$ of a $3$-manifold with non-empty
boundary, strictly less number of rectangles compared to
Casson-Gordon's rectangle condition implies strong irreducibility.
\end{proposition}

\section{Parity condition}

Let $H_1\cup_S H_2$ be a genus $g\ge 2$ Heegaard splitting of a
$3$-manifold $M$.  Let $\{D_1,D_2,\cdots,D_{3g-3}\}$ and
$\{E_1,E_2,\cdots,E_{3g-3}\}$
  be collections of essential disks of  $H_1$ and $H_2$ respectively
  giving pants decomposition of $S$.

\begin{definition}
We say that $H_1\cup_S H_2$ satisfies the even parity condition
 if $|D_i\cap E_j|\equiv 0 \pmod{2}$ for all the pairs $(i,j)$.
\end{definition}

 First we give an example of an irreducible manifold which has a
 Heegaard diagram satisfying the even parity condition. The example is a
 weakly reducible genus three Heegaard splitting of ${\text
 {(torus)}}\times S^1$ which gave inspiration for the parity
 condition. Since $\pi_1({\text{(torus)}}\times S^1)$ has rank $3$,
 genus three Heegaard splitting is of minimal genus, hence
 irreducible.

Consider ${\text{(torus)}}\times \{{\text{pt}}\}$ in ${\text
 {(torus)}}\times S^1$. Remove small open disk ${\text{int}}(D)$ from
 ${\text{(torus)}}\times \{{\text{pt}}\}$ and take its
 product neighborhood
 $({\text{(torus)}}\times \{{\text{pt}}\} -{\text{int}}(D))\times
 [0,1]$. It is a genus two handlebody. Attach a $1$-handle to
$({\text{(torus)}}\times \{{\text{pt}}\} -{\text{int}}(D))\times
 [0,1]$ along two disk
 $D_1\subset ({\text{(torus)}}\times \{{\text{pt}}\} -{\text{int}}(D))\times
 \{0\}$ and
 $D_2\subset({\text{(torus)}}\times \{{\text{pt}}\} -{\text{int}}(D))\times
 \{1\}$. The result is a genus three handlebody
 $H_1=({\text{(torus)}}\times \{{\text{pt}}\} -{\text{int}}(D))\times
 [0,1]\underset{D_1\cup D_2}{\cup} 1{\text{-handle}}$.
 The exterior $H_2={\text{cl}}({\text{(torus)}}\times S^1-H_1)$ is also a genus three
 handlebody and this gives a Heegaard splitting
 ${\text{(torus)}}\times S^1=H_1 \cup H_2$. This kind of Heegaard
 splitting is well understood, for example in \cite{Schultens}.

\begin{figure}[h]
   \centerline{\includegraphics[width=6cm]{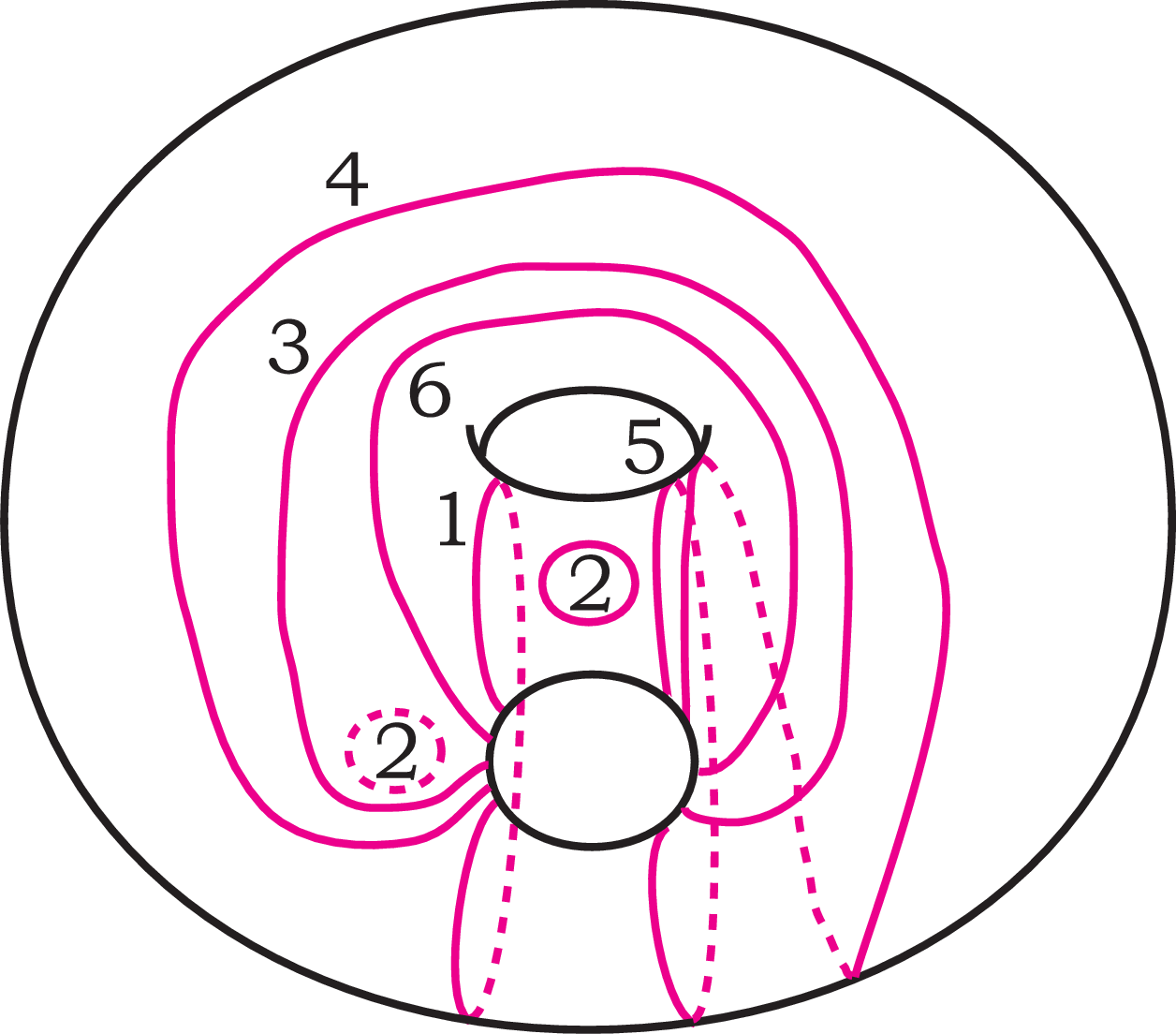}}
    \caption{essential disks in $H_1$ giving a pants decomposition}
\end{figure}

 Take a collection of essential disks with labels $1,2,3,4,5,6$ in
 $H_1$ giving a pants decomposition as in the Figure 3.
 Figure 3. shows essential arcs $1,3,4,5,6$ in a once punctured
 torus, but regard the surface and curves as being taken products with
 $[0,1]$. The two curves with label $2$ indicates the positions to
 which a $1$-handle is attached.

\begin{figure}[h]
   \centerline{\includegraphics[width=10cm]{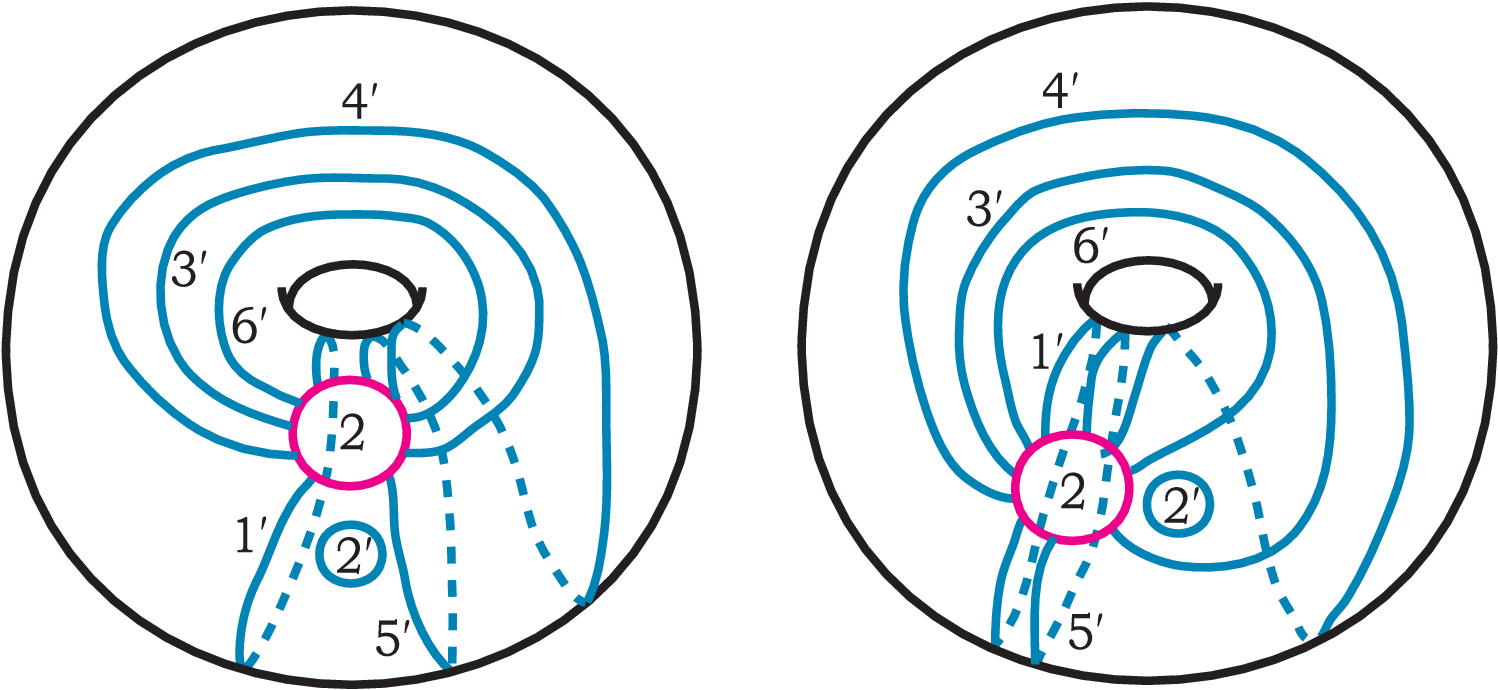}}
    \caption{Intersections of essential disks $1', 2', 3', 4', 5', 6'$
    of $H_2$ with
    ${\text{(torus)}}\times \{0\}$ and
    ${\text{(torus)}}\times \{1\}$}
\end{figure}

In $H_2$, a collection of essential disks with labels
$1',2',3',4',5',6'$ giving a pants decomposition is taken as in
the Figure 4. The left one in Figure 4. is $H_2\cap
({\text{(torus)}}\times \{0\})$ and the right one is $H_2\cap
({\text{(torus)}}\times \{1\})$. Each of arcs $1', 3', 4', 5', 6'$
in the left one is connected to corresponding one in the right
passing through the disk with label $2$. A $1$-handle is attached
along the disks with label $2'$.

Now we have information about the boundaries of all $12$ essential
disks of this specific decomposition and we can check all pairs of
intersections. For each of disks $1', 2',3',4',5',6'$, the
sequence of intersections with $1,2,3,4,5,6$ is as follows
ignoring the orientations and starting points.
$$ 1' = 1436123462$$
$$ 2' = 1546351436$$
$$ 3' = 1452341532$$
$$ 4' = 1635261532$$
$$ 5' = 5364523462$$
$$ 6' = 1452615462$$

So we can see that it satisfies the even parity condition, hance
irreducible.

\vspace{0.2cm}
\begin{proof}{\it (of Theorem 1.1)}
It is well known that reducible Heegaard splitting of an
irreducible manifold is stabilized. So it suffices to show that
$H_1\cup_S H_2$ is non-stabilized.

\begin{figure}[h]
   \centerline{\includegraphics[width=10cm]{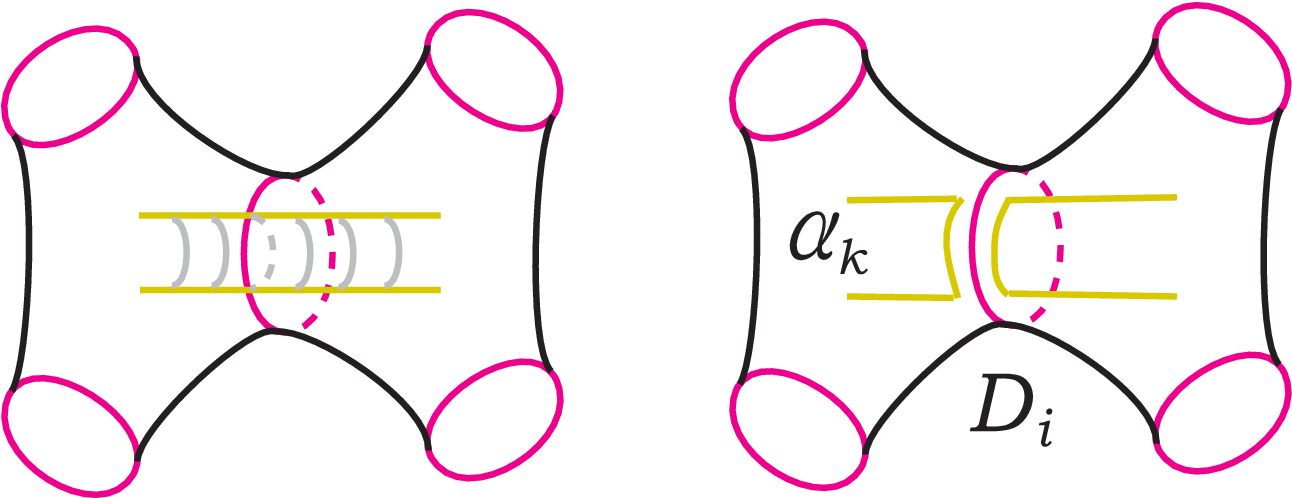}}
    \caption{$\alpha_k$ lives in a pair of pants.}
\end{figure}

Suppose that $H_1\cup_S H_2$ is stabilized. Then there exist
essential disks $D$ in $H_1$ and $E$ in $H_2$ such that $|D\cap
E|=1$. We may assume that the intersection $D\cap
(\underset{i=1}{\overset{3g-3}{\bigcup}} D_i)$ is a collection of
arcs. Cut $D$ by $(\underset{i=1}{\overset{3g-3}{\bigcup}} D_i)$.
Then $D$ is divided into subdisks. For any arc, say $\gamma$, of
intersection $D\cap D_i$, two copies of $\gamma$, $\gamma_1$ and
$\gamma_2$ are created on both sides of $D_i$ which are parallel
to each other. Connect two endpoints of $\gamma_1$ and also
connect two endpoints of $\gamma_2$ by arcs in $S$ that are
parallel and in opposite sides of $D_i$ to each other as in the
Figure 5. We do this cut-and-connect operation for all the arcs
$D\cap (\underset{i=1}{\overset{3g-3}{\bigcup}} D_i)$. Let
$\{\alpha_k\}$ be a collection of loops thence obtained from
$\partial D$. Note that each $\alpha_k$ lives in a pair of pants.
Some $\alpha_k$ would be isotopic to $\partial D_{i_k}$ and some
other $\alpha_k$ would possibly be a trivial loop.

Similarly, from $\partial E$ we obtain a collection of loops
$\{\beta_k\}$. Some $\beta_k$ would be isotopic to $\partial
E_{j_k}$ and some other $\beta_k$ would possibly be a trivial
loop.

First we consider the parity of $|\partial D\cap \partial E_j|$
for each $j$ ($j=1,2,\cdots,3g-3$) which will be used in the
below. Its parity is equivalent to
$\underset{k}{\sum}\,|\alpha_k\cap
\partial E_j| \pmod{2}$ since in the above cut-and-connect operation
two parallel copies $\gamma_1$ and $\gamma_2$ were created. It is
again equivalent to $\underset{k}{\sum}\,|\partial D_{i_k}\cap
\partial E_j|+\sum | {\text {(trivial loop)}} \cap \partial E_j|\pmod{2}$.
By the hypothesis of even parity condition, it is even. By similar
arguments, we have the following equalities in $\pmod{2}$.

$$|\partial D\cap \partial E|
\equiv \underset{k}{\sum}|\partial D\cap \beta_k| \equiv
\underset{k}{\sum}|\partial D\cap \partial E_{j_k}|+\sum |\partial
D\cap {\text {(trivial loop)}}| \equiv 0$$

This is a contradiction since $|D\cap E|=1$. So we conclude that
$H_1\cup_S H_2$ is irreducible.
\end{proof}

\section{Examples obtained by a single Dehn twist}
It is known that there exist $3$-manifolds which have infinitely
many distinct strongly irreducible Heegaard splittings. Many of
the known examples are obtained from a given Heegaard splitting by
Dehn twisting the surface several times. For example, see
\cite{CG} and \cite{MSS}. In this section, we obtain
non-stabilized Heegaard splittings by just a single Dehn twist, if
the given splitting satisfies the even parity condition.

\begin{lemma}
Suppose a genus $g\ge 2$ Heegaard splitting $H_1\cup_S H_2$ and
collections of essential disks $\{D_1,D_2,\cdots,D_{3g-3}\}$ and
$\{E_1,E_2,\cdots,E_{3g-3}\}$ satisfy the even parity condition.
Let $\alpha$ be a simple closed cure in $S$ such that $|\alpha\cap
E_j|\equiv 0\pmod{2}$ for all ($j=1,2,\cdots,3g-3$).

 If we alter $\{D_1,D_2,\cdots,D_{3g-3}\}$ to
 $\{D'_1,D'_2,\cdots,D'_{3g-3}\}$
by Dehn twisting $\partial H_1$ along $\alpha$ and don't change
$\{E_1,E_2,\cdots,E_{3g-3}\}$, then
$\{D'_1,D'_2,\cdots,D'_{3g-3}\}$ and $\{E_1,E_2,\cdots,E_{3g-3}\}$
satisfy the even parity condition.
\end{lemma}
\begin{proof}
By Dehn twisting $\partial H_1$ along $\alpha$, we can calculate
that $|D'_i\cap E_j|$ is equal to $|D_i\cap E_j|+|D_i\cap
\alpha|\cdot|\alpha\cap E_j|$. Since both $|D_i\cap E_j|$ and
$|\alpha\cap E_j|$ are even, $|D'_i\cap E_j|$ is an even number.
So we conclude that $\{D'_1,D'_2,\cdots,D'_{3g-3}\}$ and
$\{E_1,E_2,\cdots,E_{3g-3}\}$ satisfy the even parity condition.
\end{proof}

Take a neighborhood $N(\alpha)$ of $\alpha$ in $H_1$. We may
assume that ${\text{cl}}(H_1-N(\alpha))$ is a handlebody of same
genus with $H_1$ and ${\text{cl}}(H_1-N(\alpha))\cap N(\alpha)$ is
an annulus. Remove $N(\alpha)$ from $H_1$ and attach a solid torus
back to ${\text{cl}}(H_1-N(\alpha))$ so that a meridian of the
attaching solid torus is mapped to $\frac{1}{1}$-slope of
$\partial N(\alpha)$. This $\frac{1}{1}$-surgery is equivalent to
that $H_1$ is changed by a Dehn twist of $\partial H_1$ along
$\alpha$. Hence we have the following result by Lemma 5.1.

\begin{theorem}
Suppose a genus $g\ge 2$ Heegaard splitting $H_1\cup_S H_2$ and
collections of essential disks $\{D_1,D_2,\cdots,D_{3g-3}\}$ and
$\{E_1,E_2,\cdots,E_{3g-3}\}$ satisfy the even parity condition.
Let $\alpha$ be a simple closed cure in $S$ such that $|\alpha\cap
E_j|\equiv 0\pmod{2}$ for all ($j=1,2,\cdots,3g-3$).

Then the manifold obtained by $\frac{1}{1}$-surgery on $\alpha$
has a non-stabilized Heegaard splitting which is obtained from $S$
by a Dehn twist along $\alpha$.
\end{theorem}

{\bf Acknowledgement}

The author would like to thank Eric Sedgwick and John Berge for
advices on previous mistake and Sangyop Lee for reading the
manuscript.

\end{document}